%% file: paper.tex
\documentclass[11pt]{article}
\input{style_for_notes.tex}
 
\title{Constrained Optimization Involving  Nonconvex $\ell_p$ Norms: Optimality Conditions, Algorithm and Convergence}
\author{Hao Wang\thanks{School of Information Science and Technology, ShanghaiTech University, Shanghai, China (E-mail: wanghao1@shanghaitech.edu.cn).},\quad Yining Gao\thanks{School of Information Science and Technology, ShanghaiTech University, Shanghai, China (E-mail: gaoyn@shanghaitech.edu.cn).},\quad Jiashan Wang\thanks{Department of Mathematics, University of Washington, Seattle, WA, USA (E-mail: jsw1119@gmail.com).},\quad and\  Hongying Liu\thanks{School of Mathematical Sciences, Beihang University, Beijing, China (E-mail: liuhongying@buaa.edu.cn).}}

\date{}

\begin{document}
\maketitle
\begin{abstract}
	This paper investigates the optimality conditions for characterizing the local minimizers of the constrained optimization problems involving an $\ell_p$ norm ($0<p<1$), which may appear in either the objective or the constraint. This kind of problem has strong applicability to a wide range of areas since the $\ell_p$ norm can promote sparse solutions. However, the nonsmooth and non-Lipschitz nature of the $\ell_p$ norm makes these problems difficult to analyze and solve. We derived the first-order necessary conditions and the sequential optimality conditions under various constraint qualifications.  We extend the iteratively reweighted algorithms for solving the unconstrained $\ell_p$ norm regularized problems to constrained cases and show the sequential optimality conditions are satisfied by this algorithm.  Global convergence is derived and the performance of the proposed algorithm is exhibited by numerical experiments. 
\end{abstract}

\section{Introduction}
 
Sparse regularization problems have attracted considerable attentions over the past decades, which 
have numerous applications in the areas including compressed sensing \cite{chartrand2007exact}, biomedical 
engineering \cite{Xu2007Lp,2018Resolution}, sensor selection \cite{2018Density} and portfolio management \cite{chen2013sparse}. 
This is because sparse solutions usually lead to better generalization performance and robustness. 
A common approach of promoting sparsity in the solution  
 is to involve a sparse-inducing term such as the $\ell_1$ norm or the $\ell_0$ norm of the variables 
 either in the objective as a penalty 
 or in the constraint. In recent years, nonconvex and/or non-Lipschitz sparsity inducing terms 
 such as the $\ell_p$ (quasi-)norm ($0<p<1$) are shown \cite{Oymak2015Sharp}  to have preferable  performance in many situations. 
 In the past decade, many works focus on designing and analyzing the algorithms for solving the unconstrained $\ell_p$ regularized problems
 \cite{Xu2010L1, ge2011note,2013Decentralized,Heng2014Optimality,2019Relating,haeser2019optimality}. 
However, when it comes to the constrained cases,  there are few works despite of its wider applicability. 
We list two examples of the constrained optimization problems involving an $\ell_p$ norm.

\emph{Example 1}  Consider the cloud radio access network (Cloud-RAN) power consumption problem~\cite{gsbf, wang2018nonconvex}, 
 which solves a group sparse problem  to induce the group sparsity for the beamformers to guide the remote radio head (RRH) selection. This group sparse problem is addressed by minimizing the mixed $\ell_p/\ell_2$-norm  with $p\in(0,1]$, yielding the following problem
\begin{equation*}
\begin{aligned}
\min_{ v }\quad & \sum_{l=1}^{L}\sqrt{\frac{\rho_l}{z_l}}\|\tilde{ v }_l\|_2^p\\
\st \quad & \sqrt{\sum_{i\neq k} \| h_k^{\sf{H}}  v _i \|_2^2 + \sigma _k^2}\leq \frac{1}{\gamma _k}\Re (h_k^{\sf{H}}  v _k)\\
& \| {\tilde{ v }}_{l}  \| _2\leq \sqrt{P_l},\ l =1,\cdots,L, k = 1,\ldots, K.
\end{aligned}\label{f.GS}
\end{equation*}
Here the Cloud-RAN architecture of this model has $L$ RRHs and $K$ single-antenna Mobile Users (MUs), where the $l$-th RRH is equipped with $N_l$ antennas. 
$ v _{lk} \in \mathbb{C}^{N_l}$ is the transmit beamforming vector from the $l$-th RRH to the $k$-th user with the group structure of transmit vectors ${\tilde{ v }}_{l}=[ v _{l1}^T,\cdots,  v _{lK}^T]^T\in \mathbb{C}^{KN_l\times 1}$. Denote the relative fronthaul link power consumption by $\rho_l$, and the inefficient of drain efficiency of the radio frequency power amplifier by $z_l$.
The channel propagation between user $k$ and RRH $l$ is denoted as $\hbf_{lk}\in \mathbb{C}^{N_l}$. 
$P_l$ is the maximum transmit power of the $l$-th RRH.
$\sigma _k$ is the noise at MU $k$.
$ {\gamma} = (\gamma_1,..., \gamma_K)^T$ is the target signal-to-interference-plus-noise ratio (SINR).
 

\emph{Example 2}. 
(The $\ell_{p}$-constrained sparse coding) In the context of sparse coding \cite{Thom2015Efficient}, the task is to reconstruct the unknown sparse code word $\bar x\in\mathbb{R}^{n}$ from the linear measurements $ {y} =  {A}\bar x +  \epsilon$, where  $ {y}\in\mathbb{R}^{m}$ represents the data with $m$ features,  $ {\epsilon}\in\mathbb{R}^{m}$ denotes the noise vector, and  $ {A}$ corresponds to the fixed dictionary that consists of $n$ atoms with respect to its columns. This problem can be formulated as 
	\begin{equation}\label{eq: lp_regression}
 	\min_{x  \in \mathbb{R}^n}~  \Vert {A}x - {y}\Vert_{2}^{2}  \qquad
	\textrm{s.t.}~  \Vert x \Vert_{p}^{p} \leq \theta,
 	\end{equation}
	where the $\ell_{p}$ ball constraint is to induce sparsity in the code word. 

In this paper, we consider  the following two general  forms of constrained nonlinear optimization with $\ell_p$ norms. The first one 
is the constrained $\ell_p$ regularized problem, meaning the $\ell_p$ norm appears in the objective as a penalty, 
\begin{equation}\label{eq.problem1}\tag{\text{$\mathscr{P}_1$}}
\min\ F(x ) := f_0(x )+ \lambda\|x \|^p_p \quad \text{ s.t. }\  f_j(x )\leq 0, \ \forall j\in\mathcal{I}; \  f_j(x )=0,\ \forall j\in\mathcal{E}. 
\end{equation}
The second one has the $\ell_p$ norm in the constraint and requires it to be smaller than a prescribed value $\theta > 0$, 
\begin{equation}\label{eq.problem2}\tag{\text{$\mathscr{P}_2$}}
\min\ f_0(x )\quad \text{ s.t. }\ \|x \|^p_p\leq \theta ;\  f_j(x )\leq 0,\ \forall j\in\mathcal{I};\  f_j(x )=0,\ \forall j\in\mathcal{E}. 
\end{equation} 
Here, $f_j:\mathbb{R}^n\to\mathbb{R}, j\in\{0\}\cup \Ical\cup \Ecal$ are continuously  
differentiable on $\mathbb{R}^n$ and  $\|x \|_p := (\sum\limits_{i=1}^n|x_i|^p)^{1/p}$  with $p\in(0,1]$.    The positive $\lambda$ is the given regularization parameter and   $ \theta$   is referred to as the radius of $\ell_p$-ball.

Despite of the advantages of nonconvex $\ell_p$ norm in promoting sparse solutions,  problems of the forms  
\eqref{eq.problem1} and \eqref{eq.problem2} are generally not easy to handle. This is largely due to the nonconvex 
and non-Lipschitz  nature of the $\ell_p$ norm which makes it difficult to characterize the optimal solutions. 
 In particular, verifiable optimality conditions are often difficult to derive, leaving it an obstacle for designing 
 efficient numerical algorithms. For example, for  \eqref{eq.problem1}, many researchers \cite{chen2010lower,lu2014iterative,wang2018nonconvex,2019Relating} tend to approximate 
 the $\ell_p$ term  by Lipschitz continuous functions and then solve for an approximate solution. 
 As for \eqref{eq.problem2},  
 to the best of our knowledge, not much has been done except the special case that only the $\ell_p$ ball 
constraint presents in the problem \cite{wang2021efficient}, meaning the projection onto the $\ell_p$ ball.

\subsection{Literature review}
The optimality conditions of the unconstrained  and inequality constrained versions of \eqref{eq.problem1} were studied 
in \cite{wang2015optimality}, which is the immediate predecessor of our work. They derived the first-order and second-order necessary conditions by assuming the ``extended'' linear independence 
constraint qualification (ELICQ) is satisfied by \eqref{eq.problem1}, meaning the LICQ is satisfied at the local minimizer in the subspace consisting of the 
nonzero variables.  They also stated that  the second-order optimality conditions can be derived by considering the reduced 
problems after fixing the zero components at a stationary point. 
   In  \cite{bian2017optimality}, 
 Bian and Chen derived a first-order necessary optimality condition using the theory of the generalized directional derivative, which is also closely related to our work. 
In particular, for the case that the constraints are all linear, Gabriel Haeser et al. \cite{haeser2019optimality}   articulated first- and second-order necessary optimality conditions for this problem based on the perturbed problem and the limits of  perturbation. Sufficient conditions for the $\epsilon$-perturbed stationary points are also presented.
As for \eqref{eq.problem2}, \cite{wang2021efficient} derived optimality conditions for the special case where only $\ell_p$ ball constraint exists using the concept 
of generalized Fr\'echet normal cone. To the best of our knowledge, there has been no study on the optimality conditions for 
more general cases of this problem.

\subsection{Contributions}

In this paper, we are interested in deriving the optimality conditions to characterize the local solutions of \eqref{eq.problem1} and \eqref{eq.problem2} 
under different constraint qualifications (CQ).  
First of all,  we analyze the basic properties of the $\ell_p$ norm and the $\ell_p$ norm ball. 
We derive the regular and general subgradients of the $\ell_p$ norm   and   
the regular and general normal of the $\ell_p$ norm ball, which indicate the 
$\ell_p$ norm is subdifferentially regular and the $\ell_p$ ball is 
Clarke regular. 
For   \eqref{eq.problem1} and \eqref{eq.problem2}, we   derive  the 
Karush-Kuhn-Tucker (KKT) conditions  and discuss the constraint qualifications that ensure that the KKT 
conditions are satisfied at a local minimizer. 
For   \eqref{eq.problem2},   we believe this is the first result.

 Recently, Andreani et al. \cite{andreani2016cone} introduced the sequential optimality conditions, namely, the approximate KKT (AKKT) conditions  for constrained smooth optimization problems, which is  
 commonly satisfied by many algorithms.  They also proposed the Cone-Continuity Property (CCP),  
 under which the AKKT conditions  implies the KKT conditions.  This is widely believed to be one of the weakest qualification under which KKT conditions  
 hold  at local minimizer.  We also define the sequential optimality conditions for   \eqref{eq.problem1} and \eqref{eq.problem2}  
 and  explore  the   constraint qualification under which the sequential conditions imply KKT conditions. 
 We believe these are much stronger results than existing ones. 
 
 To demonstrate  the applicability of the proposed sequential optimality conditions, we  extend the well-known iteratively reweighted 
 algorithms for solving unconstrained $\ell_p$-regularized problem to general constrained cases and show that those conditions 
 are satisfied at the limit  points of the generated iterates.  Therefore,  under the proposed constraint qualification, 
 the limit points  satisfy the KKT conditions.

%
%
%

 \subsection{Notation and preliminary}
 
 We use $0$ as the vector filled with all zeros of appropriate size.  
For $\Ncal \subset \{1,\ldots, n\}$ and $x \in \mathbb{R}^n$,  let  $\mathbb{R}^{|\Ncal |}$ be  the reduced subspace of $\mathbb{R}^n$ that consists of 
the components $x_i, i\in \Ncal$, and denote  $x _{\Ncal} \in \mathbb{R}^{|\Ncal |}$ as the subvector of $x $   containing  
the elements  $x_i, i\in \Ncal$.  Let $\Zcal = \{1, \ldots, n\} \setminus \Ncal$.  
For a differentiable $f: \mathbb{R}^n \to \mathbb{R}$,  let $\nabla_{\Ncal} f( x)$ be the vector consisting of 
$\nabla_i f( x), i\in \Ncal$. 
In  \eqref{eq.problem1} and \eqref{eq.problem2}, define the set   
 \begin{equation*} 
 \Gamma =\{x \ |\ f_j(x )\leq 0,\ \forall j\in\mathcal{I};\quad f_j(x )=0,\ \forall j\in\mathcal{E}\},
 \end{equation*}
and  the index set of active inequalities   by 
 $\mathcal{A}(x )=\{j\ |\ f_j(x )=0,\ j\in\mathcal{I}\}.$  
The sets of zeros and nonzeros in $x\in\mathbb{R}^n$ are defined  as 
 $$\Zcal (x )=\{i\ |\ x_i=0\}\  \text{ and }\  \Ncal(x )=\{i\ |\ x_i\neq 0\}.$$
 For simplicity,  we use shorthands  $\bar\Ncal = \Ncal(\bar x)$, $\bar\Zcal = \Zcal(\bar x)$  and $\Ncal^k = \Ncal(x^k)$, $\Zcal^k = \Zcal(x^k)$.

For $x \in\mathbb{R}^n$, define $\mathbb{B}(x ; \delta )=\{z \mid \| {z} -x \|_2 \le \delta \}$. 
For any cone $K\subset \mathbb{R}^n$,  the polar of $K$ is defined to be the cone  
$$K^* := \{ v \in\mathbb{R}^n\mid \langle  v, w \rangle\leq 0 \text{ for all } w\in{K}\}.$$
 For   $C\subset \mathbb{R}^n$, its  horizon cone is defined by 
\[ C^\infty = \begin{cases}  \{ x \mid \exists x^\nu \in C, \lambda^\nu \searrow 0, \text{ with } \lambda^\nu x^\nu \to x\}& \text{when } C\ne \emptyset, \\
\{0\} & \text{when } C=\emptyset. 
\end{cases} 
\]
Another operation on $C$ is the smallest cone containing $C$, namely the positive hull of $C$, which is defined as 
$ \text{pos }C=\{0\} \cup\{\lambda x\mid x\in C, \lambda > 0\}.$ 
A vector $w\in\mathbb{R}^n$ is tangent to a set $C\subset \mathbb{R}^n$ at $\bar  x \in C$, written $w\in T_C(\bar x)$, if 
$(x^\nu - \bar x )/\tau^\nu \to w  \text{   for some  }   x^\nu \xrightarrow[C]{} \bar x,  \tau^\nu \searrow 0.$
The interior and the boundary of a set $C\subset \mathbb{R}^n$ is denoted as $int$ $C$ and $\partial C$, respectively. 


\begin{definition}[\cite{rockafellar2009variational} Definition 6.3, 6.4] 
	\begin{enumerate}
		\item[(a)] Let $C\subset \mathbb{R}^n$ and $\bar x\in C$. A vector 
$v$ is a regular normal to $C$ at $\bar x$, written $v\in \widehat N_C(\bar x)$, if 
\[ \limsup\limits_{x\xrightarrow[C]{} \bar x, x\ne\bar x} \frac{\langle v, x-\bar x\rangle}{\|x-\bar x\|} \le 0.\]
It is a (general) normal to $C$ at $\bar x$, written $v\in N_C(\bar x)$, if 
there are sequences $x^\nu\xrightarrow[C]{} \bar x$ and $v^\nu \to v$ with $v^\nu \in \widehat N_C(x^\nu)$. 
We call $ \widehat N_C(\bar x)$ the regular normal cone and  $ N_C(\bar x)$ the normal cone to $C$. 
\item[(b)]A set $C\subset \mathbb{R}^n$ is Clarke regular at $\bar x \in C$  if it is locally closed at $\bar x$ and 
$N_C(\bar x) = \widehat N_C(\bar x)$. 
	\end{enumerate}
\end{definition} 
For a nonempty convex $C\subseteq \mathbb{R}^n$ and $\bar x \in C$,   
$\widehat N_C(\bar x) = N_C(\bar x ) = \{v \mid \langle v,  z -\bar  x  \rangle\leq 0,\text{ for all } z \in C\}.$

\begin{definition}[\cite{rockafellar2009variational} Definition 8.3]  Consider a function $f: \mathbb{R}^n \to \bar{\mathbb{R}} = \mathbb{R}\cup \{+\infty\}$ and $f(\bar x) < \infty$. For a vector $v\in\mathbb{R}^n$, one says that 
\begin{enumerate}
\item[(a)] $v$ is a regular subgradient of $f$ at $\bar x$, written $v\in \widehat\partial f(\bar x)$, if 
\[ f(x) \ge f(\bar x) + \langle v, x-\bar x\rangle + o(|x-\bar x|);\]
\item[(b)] $v$ is a (general) subgradient of $f$ at $\bar x$, written $v\in \partial f(\bar x)$, if there are sequence 
$x^\nu \xrightarrow[f]{} \bar x$ and $v^\nu \in \widehat\partial f(x^\nu)$ with $v^\nu \to v$; 
\item[(c)] $v$ is a horizon subgradient of $f$ at $\bar x$, written $v\in \partial^\infty f(\bar x)$, if there are 
sequence $x^\nu \xrightarrow[f]{} \bar x$ and $v^\nu \in \widehat\partial f(x^\nu)$ with $\lambda^\nu v^\nu \to v$ 
for some sequence $\lambda^\nu \searrow 0$. 
\end{enumerate}
\end{definition} 
For $f: \mathbb{R}^n\to\mathbb{R}$, the epigraph of $f$ is the set 
$ \text{epi } f:=\{(x,\alpha) \in \mathbb{R}^n \times \mathbb{R} \mid \alpha \ge f(x)\}.$

\begin{definition}[\cite{rockafellar2009variational} Definition 7.25] 
	 A function $f: \mathbb{R}^n\to \bar{\mathbb{R}}$ is called 
	subdifferentially regular at $\bar x$ if $f(\bar x)$ is finite and $\text{epi }f$ is Clarke regular at $(\bar x, f(\bar x))$ as 
	a subset of $\mathbb{R}^n \times \mathbb{R}$. 
	\end{definition}

\section{First-order necessary optimality conditions}
In this section, we present the   first-order necessary  optimality  conditions for \eqref{eq.problem1} and \eqref{eq.problem2}. Before proceeding to the optimality conditions, we provide some basic properties. 
\subsection{Basic Properties}
Denote  $\phi(x) = \| x \|_p^p$ and  the $\ell_p$ norm ball $\Theta  := \{ x \in \mathbb{R}^n \mid  \phi(x) \le \theta \}$. In this subsection, we provide basic properties about $\phi(x)  $ and $\Theta$. In particular, we derive regular and general subgradients of $\phi$ and the regular and the general normal cones of $\Theta$, and then show that the $\phi$ is subdifferentially regular and $\Theta$ is  Clarke regular on $\mathbb{R}^n$. 

The regular, general, and horizon subgradients of $\phi$ can be calculated as follows. 
\begin{theorem}\label{thm.regular1} For any $\bar x \in \mathbb{R}^n$,  it holds that 
\begin{align}
 \partial \phi(\bar x) = \widehat \partial \phi(\bar x) = \ & \{ v \in\mathbb{R}^n \mid v_j = \sign(x_j)p|\bar x_j|^{p-1}, \ j\in\bar\Ncal\},\label{partial phi} \\
 [\widehat \partial \phi(\bar x)]^\infty = \partial^\infty \phi(\bar x) = \ & \{ v \in\mathbb{R}^n \mid v_j = 0, \ j\in\bar\Ncal\}.\label{partial phi inf}
 \end{align}
 Therefore, $\phi$ is subdifferentially regular at every $x\in\mathbb{R}^n$. 
\end{theorem} 
\begin{proof} 
We first consider $|\cdot |^p$ on $\mathbb{R}$.  If $\bar x\ne 0$, then $\partial |\bar x|^p = \widehat\partial |\bar x|^p = 
\{ \nabla |\bar x|^p \} = \{  \text{sign}(\bar x) p |\bar x|^{p-1} \}$. On the other hand, 
$\lim\limits_{x\to 0, x\ne 0}\frac{|x|^p}{|x|} = + \infty$, implying 
$\liminf\limits_{x\to 0, x\ne 0}\frac{|x|^p - |0|^p - v(x-0)}{|x-0|} \ge 0$ for any $v\in\mathbb{R}$. 
Therefore, if $\bar x=0$,  it follows from the definition of regular subgradient that 
  $\mathbb{R}\subset \widehat\partial |0|^p \subset \partial |0|^p \subset \mathbb{R}$. Hence, 
  $\widehat\partial |0|^p = \partial |0|^p =   \mathbb{R}$.  
 By [\cite{rockafellar2009variational} Proposition 10.5], we have  
 \[ \widehat\partial \phi(\bar x) = \partial \phi(\bar x) = \partial |\bar x_1|^p \times \ldots \times  \partial |\bar x_n|^p =   \{ v \in\mathbb{R}^n \mid v_j = \text{sign}(\bar x_j)p|\bar x_j|^{p-1}, \ j\in\bar\Ncal\}.\]
 This proves  \eqref{partial phi}. 
 
 By the definition of horizon cone and \eqref{partial phi}, it is obvious that $ [\widehat \partial \phi(\bar x)]^\infty = \{ v \in\mathbb{R}^n \mid v_j = 0, \ j\in\bar\Ncal\}$.  We next prove 
 $ \partial^\infty \phi(\bar x)  =  [\widehat \partial \phi(\bar x)]^\infty$. 
 
 For any $v \in [\widehat \partial \phi(\bar x)]^\infty$ and $\{\lambda^\nu\}\searrow 0$,   we can select sequence $x^\nu \xrightarrow[\phi]{} \bar x$   
 such that   $\Ncal(x^\nu) = \bar \Ncal$. Let $v^\nu_j = v_j/\lambda^\nu$.  
From \eqref{partial phi},  this means $v^\nu \in \widehat\partial \phi(x^\nu)$ and $\lambda^\nu v^\nu\to v$. 
Therefore, $v\in \partial^\infty \phi(x)$. 
  
On the other hand, for $x^\nu$ sufficiently  close to $\bar x$,  it holds that $\bar \Ncal \subset \Ncal(x^\nu)$. 
Therefore, by  \eqref{partial phi},   
$v_j^\nu = \text{sign}(\bar x_j)p|\bar x_j^\nu|^{p-1}, j\in \bar \Ncal$ for any $v^\nu \in \widehat\partial \phi(x^\nu)$. 
 Hence, for any sequence $\{\lambda^\nu\} \searrow 0$,  $\lambda^\nu v_j^\nu \to 0$,  
 $j \in \bar \Ncal$,  it holds that  $v_j = 0, j\in \bar \Ncal$ for any $v\in \partial^\infty \phi(\bar x)$, or, equivalently,  
 $\partial^\infty \phi(\bar x)  \subset  [\widehat \partial \phi(\bar x)]^\infty$.  Overall, we have shown 
 that $\partial^\infty \phi(\bar x) =   [\widehat \partial \phi(\bar x)]^\infty$. 
 
 It then follows from \cite[Corollary 8.11]{rockafellar2009variational} that $\phi$ is  subdifferentially regular at any $x\in\mathbb{R}^n$. 
 \end{proof}

The regular and general normal vectors can be calculated as follows. 

\begin{theorem} \label{thm.regular2}
For   any $\bar x\in \Theta$, $N_\Theta(\bar x) = \text{pos }\partial \phi(\bar x) \cup \partial^\infty \phi(\bar x)$, i.e.,
 \begin{equation}\label{normal cones1} N_\Theta(\bar x) =   \begin{cases} 
  \{  v \in\mathbb{R}^n \mid v_j = \lambda \sign(\bar x_j)p|\bar x_j|^{p-1},  j\in\bar \Ncal; \  \lambda  \ge 0\} & \text{if }\bar  x\in\partial \Theta,\\
  \{0\} & \text{if } \bar x\in \text{int } \Theta.
  \end{cases}\end{equation} 
  Furthermore,  $\Theta$ is  Clarke regular at any $\bar x\in \Theta$, i.e., $\widehat N_\Theta(\bar x) = N_\Theta(\bar x)$. 
\end{theorem} 

\begin{proof}  We only prove the case that $\bar x\in\partial \Theta$ since the other is trivial. We have $\bar x \neq 0$ and $0\not\in\partial \phi (\bar x)$. Together with \Cref{thm.regular1} and \cite[Proposition 10.3]{rockafellar2009variational}, it holds that 
\[\widehat N_\Theta(\bar x) = N_\Theta(\bar x) = \text{pos }\partial \phi(\bar x) \cup \partial^\infty \phi(\bar x) \]
and $\Theta$ is Clarke regular at $\bar x$.
By the definition of \text{pos} and \eqref{partial phi},
 \[  \text{pos }\partial \phi(\bar x) = \{0\}\cup \{ \lambda v \in\mathbb{R}^n \mid v_j = \text{sign}(\bar x_j)p|\bar x_j|^{p-1},  j\in\bar \Ncal;  \  \lambda  >  0\}. \]
Therefore, it holds that 
\[ N_\Theta(\bar x) = \text{pos }\partial \phi(\bar x) \cup \partial^\infty \phi(\bar x) = \{   v \in\mathbb{R}^n \mid v_j = \lambda \text{sign}(\bar x_j)p|\bar x_j|^{p-1},  j\in\bar \Ncal; \  \lambda  \ge  0\}. \]
%
 \end{proof}

From \cite[Theorem 8.15]{rockafellar2009variational}, we have the following first-order necessary condition for \eqref{eq.problem1}. 
 For \eqref{eq.problem2}, we only focus on the local minimizers $\bar x$ on the boundary of the $\ell_p$ ball, i.e., $\|\bar x\|_p^p = \theta$; otherwise, the characterization of local minimizers 
 reverts to the case of traditional constrained nonlinear problems. 

\begin{theorem} \label{basic.thm}  
Suppose $f_0$ is differentiable over $\Gamma$.  The following statements hold true. 
\begin{enumerate}
\item[(a)] If $\partial^\infty \phi(\bar x)$ contains no vector $v\ne 0$ such that 
$-v\in N_\Gamma(\bar x)$, then a necessary condition for $\bar x $ to be locally optimal for \eqref{eq.problem1} is 
\begin{equation}\label{lpc.uncons}
 0 \in \nabla f_0(\bar x) +\lambda \partial \phi(\bar x) + N_\Gamma(\bar x).
 \end{equation}
\item[(b)] 
	Suppose that $\bar x$ is a local minimizer of \eqref{eq.problem2} with $\bar x\in \partial \Theta$. Then 
	\begin{equation}\label{lpc:geometric}
	 -\nabla f_0(\bar x )\in \widehat N_{\Theta\cap\Gamma}(\bar x) \subset N_{\Theta\cap\Gamma}(\bar x).
	 \end{equation}
\end{enumerate}
 \end{theorem}
\begin{figure}[htbp]
  \centering
  \includegraphics[width=0.6\linewidth]{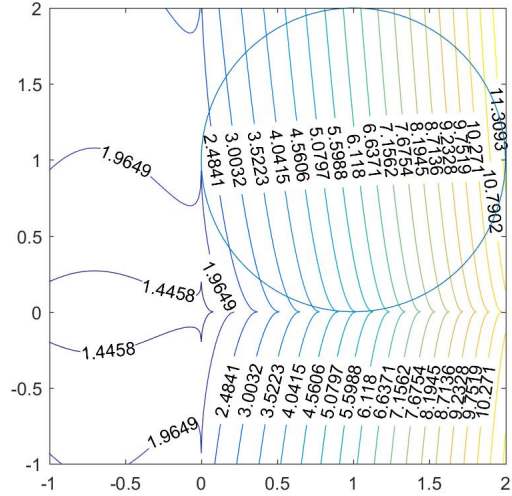}
  \caption{A conterexample for \Cref{basic.thm}.For \eqref{eq.problem1}, the contour of $\ell_p\ (p = 0.5)$ regularization problem with $F(x) = (x_1+1)^2  + \|x\|^p_p$ and $\Gamma=\{(x_1-1)^2+(x_2-1)^2\leq 1\}$. } 
\end{figure}  
Obviously, to find an optimal solution, we can only focus on the top semicircle of the given ball, so that the original 
problem is equivalent to 
  $$\min\ (x_1+1)^2+\sqrt{x_1}+\sqrt{1-\sqrt{-x_1^2+2x_1}} \text{ s.t. } 0 \leq x_1 \leq 2$$
The derivative over the domain is always positive, therefore $x^*=(0,1)$ is a global minimizer. However,  there exist $v\in\{\nu\in\mathbb{R}^2\ |\ \nu_1>0,\nu_2=0\}$, such that $\partial^\infty(\bar x)\ni v \neq 0$ and $-v \in N_\Gamma(\bar x)$. In this case,  one can see 
\eqref{lpc.uncons} does not hold at $x^*$. 
 

 \subsection{Optimality conditions for \eqref{eq.problem1}}

To make condition \eqref{lpc.uncons} for \eqref{eq.problem1} 
 informative, we need to clarify when $-v\in N_\Gamma(\bar x)$ happens and 
 how to calculate the elements in $N_\Gamma(\bar x)$.  
 For this purpose, we define the following extended 
Mangasarian-Fromovitz constraint qualification (EMFCQ). The EMFCQ holds at   $\bar x \in \Gamma$ for $\Gamma$ if the subvectors $\nabla_{\bar\Ncal} f_j(\bar x)  $, $j \in\Ecal\cup \bar\Acal$ are linearly independent and there exists $ {d}\in\mathbb{R}^{|\bar\Ncal|}$ such that 
	\begin{equation}\label{mfcq.condition} 
		\begin{aligned}
			\langle \nabla_{\bar \Ncal} f_j(\bar x), {d}\rangle =0,\ j\in\mathcal{E}\ \text{ and } \  \langle \nabla_{\bar \Ncal} f_j(\bar x), {d}\rangle <0,\ j\in\bar\Acal.
		\end{aligned}		
	\end{equation}
 Obviously, the EMFCQ is a weaker condition than the ELICQ proposed in \cite{wang2015optimality}. Moreover, 
 if the EMFCQ  holds at $\bar x\in\Gamma$ for $\Gamma$, then the 
MFCQ   holds naturally  true at  $\bar x$ for $\Gamma$; 
letting   \begin{equation}\label{normal mfcq}
\Lambda_\Gamma(\bar x) = \{  \sum_{j\in\Ecal\cup\bar\Acal } y_j \nabla f_j(\bar x) \mid y_j \ge 0, j \in \bar \Acal\}, 
  \end{equation} 
 we have from   \cite[Theorem 6.14]{rockafellar2009variational} 
 that $\Gamma$ is regular at $\bar x$ and $N_\Gamma(\bar x) = \Lambda_\Gamma(\bar x)$.

  \begin{theorem}\label{optimality.p1} 
 \begin{enumerate}
 \item[(a)]  Suppose the EMFCQ is satisfied  at   $\bar x\in\Gamma$ for $\Gamma$. Then  $\partial^\infty \phi(\bar x)$ contains no vector $v\ne 0$ such that 
$-v\in N_\Gamma(\bar x)$.  
Furthermore,  a necessary condition for $\bar x $ to be locally optimal for \eqref{eq.problem1} is that there exist $y_j,\ j\in\Ecal$ and $y_j\geq 0,\ j\in\Acal$ such that
\begin{equation}\label{kkt1 for p1}
   \nabla_i f_0(\bar x) + \lambda p\sign(\bar x_i)|\bar x_i|^{p-1} + \sum_{j\in\Ecal\cup\bar\Acal } y_j \nabla_i f_j(\bar x) = 0, \ i\in\bar{\Ncal}.
   \end{equation}
   \item[(b)]  Suppose $\Gamma$ is closed and convex in \eqref{eq.problem1}. If $\bar x$ is a local minimizer of \eqref{eq.problem1} 
   and  $\partial^\infty \phi(\bar x)$ contains no vector $v\ne 0$ such that 
$-v\in N_\Gamma(\bar x)$,    
  then it holds that 
\begin{equation}\label{eq.convexset} 
  \nabla_i f_0(\bar x) + \lambda p\sign(\bar x_i) |\bar x_i|^{p-1} + v_i = 0,  i \in \bar\Ncal;  \   v\in N_\Gamma(\bar x).
 \end{equation}

   \end{enumerate}
 \end{theorem} 
 
 \begin{proof} (a) 
Assume by contradiction that there exists nonzero $v\in  \partial^\infty \phi(\bar x)$ such that $-v\in N_\Gamma(\bar x)$;  then 
it follows from \Cref{thm.regular1} and \eqref{normal mfcq} that 
 \begin{equation}\label{tmp.mfc}    \sum_{j\in\Ecal\cup\bar\Acal}  y_j \nabla_{\bar \Ncal} f_j(\bar x) = -v_{\bar\Ncal} = 0;    \sum_{j\in\Ecal\cup\bar\Acal}  y_j \nabla_{\bar \Zcal} f_j(\bar x) = -v_{\bar\Zcal} \ne 0; \  y_j\ge0, j\in\bar\Acal.
 \end{equation} 
 Since EMFCQ holds true at $\bar x\in \Gamma$,  the dual form   \cite{Solodov10} of condition \eqref{mfcq.condition} tells that  $y_j=0, j\in\Ecal\cup\bar\Acal$ is the {unique} solution 
 of the  the system 
 \[   \sum_{j\in\Ecal\cup\bar\Acal}  y_j \nabla f_j(\bar x) = 0,\  y_j\ge0, j\in\bar\Acal.\]
It follows that $ \sum_{j\in\Ecal\cup\bar\Acal}  y_j \nabla_{\bar \Zcal} f_j(\bar x) = 0$, contradicting \eqref{tmp.mfc}. 
Therefore, for any nonzero $v\in \partial^\infty \phi(\bar x)$,  $-v\notin N_\Gamma(\bar x)$. 
 From \Cref{basic.thm}, at a local optimal solution  $\bar x$ of  \eqref{eq.problem1}, \eqref{kkt1 for p1} is satisfied.

(b)  This is trivially true  from    \Cref{basic.thm}.  
%
 \end{proof}

 We call the conditions \eqref{kkt1 for p1} the   Karush-Kuhn-Tucker (KKT) conditions for \eqref{eq.problem1}.  Using the notation of $\Lambda_\Gamma$, 
 it can also be equivalently written as 
 \begin{equation}\label{kkt.1.lambda}
 - \nabla f_0(\bar x) \in \lambda\partial \phi(\bar x)+ \Lambda_\Gamma(\bar x).
 \end{equation}

\subsection{Optimality conditions for \eqref{eq.problem2}}

 We also consider other verifiable forms of condition \eqref{lpc:geometric} if some constraint qualification is satisfied at $\bar x$. 
%
 For $\bar x \in \Theta \cap\Gamma$, define   the extended linearized cone $\Upsilon_{\Theta\cap \Gamma}(\bar x)$ as: 
\[  
\Upsilon_{\Theta\cap\Gamma}(\bar x) :=  \{ d \in \mathbb{R}^n \mid \langle v, d  \rangle \le 0, \forall v\in\partial \phi(\bar x);   
\langle \nabla  f_j(\bar x), d  \rangle = 0, j\in\Ecal; \ 
\langle \nabla f_j(\bar x), d  \rangle \le 0, j\in \bar\Acal\}. 
\]
Obviously,  
\[
\begin{aligned}
N_\Theta(\bar x) + \Lambda_\Gamma(\bar x) = &   \Upsilon_{\Theta\cap\Gamma}(\bar x)^* \\ 
 = &  \{ v \in \mathbb{R}^n \mid v_i =  y_0p\sign(\bar x_i) |\bar x_i|^{p-1}  +  \sum_{ j \in\bar\Acal\cup\Ecal } y_j \nabla_i f_j(\bar x), i\in\bar\Ncal; \   y_j \ge 0, j\in \{0\}\cup\bar \Acal \}.
\end{aligned} \]
It follows from  \cite[Theorem 6.14]{rockafellar2009variational} that 
		 $\Upsilon_{\Theta\cap\Gamma}(\bar x)^*     \subset   \widehat N_{{\Theta}\cap\Gamma}(\bar x ).$  
		 Hence, we have the following result. 
 \begin{proposition}\label{suff for nec} For $\bar x \in \Gamma \cap \Theta $ with $\bar x \in \partial \Theta$, 
$ \Upsilon_{\Theta\cap\Gamma}(\bar x)^*\subset \widehat N_{\Theta\cap\Gamma}(\bar x)  $. Therefore, if $-\nabla f_0(\bar x)\in \Upsilon_{\Theta\cap\Gamma}(\bar x)^*$, meaning that there exist $y_j \ge 0, j\in \{0\}\cup\bar \Acal $  such that 
\begin{equation*}  \nabla f_0(\bar x) +  y_0p\text{sign}(\bar x_i) |\bar x_i|^{p-1}  +  \sum_{ j \in\bar\Acal\cup\Ecal } y_j \nabla_i f_j(\bar x) = 0, i\in\bar\Ncal,   
 \end{equation*}
then the first-order necessary condition \eqref{lpc:geometric} is satisfied at $\bar x$. 
\end{proposition}

 	The EMFCQ for \eqref{eq.problem2} holds at   $\bar x \in \partial\Theta\cap\Gamma$   if the subvectors 
	$ \text{sign}(x_{\bar\Ncal})|x_{\bar\Ncal}|^{p-1}$,  $\nabla_{\bar\Ncal} f_j(\bar x)  $, $j \in\Ecal\cup \bar\Acal$ are linearly independent and there exists $ {d}\in\mathbb{R}^{|\bar\Ncal|}$ such that 
	\begin{equation}\label{mfcq.condition2} 
		\begin{aligned}
		\langle p\text{sign}(x_{\bar\Ncal})|x_{\bar\Ncal}|^{p-1}, d\rangle < 0, 	\langle \nabla_{\bar \Ncal} f_j(\bar x), {d}\rangle =0,\ j\in\mathcal{E}\ \text{ and } \  \langle \nabla_{\bar \Ncal} f_j(\bar x), {d}\rangle <0,\ j\in\bar\Acal.
		\end{aligned}		
	\end{equation}
Equivalently, the dual form  of  EMFCQ   
  for \eqref{eq.problem2}  holds at $\bar x\in\Theta\cap\Gamma$ if  $y_j = 0, j\in \{0\}\cup\bar \Acal$ is the {unique} solution of 
\[  \{  y_0p\text{sign}(\bar x_i) |\bar x_i|^{p-1}  +  \sum_{ j \in\bar\Acal\cup\Ecal } y_j \nabla_i f_j(\bar x) = 0, i\in\bar\Acal; \   y_j \ge 0, j\in \{0\}\cup\bar \Acal \}.\]


We now state the necessary optimality conditions for \eqref{eq.problem2}. 

\begin{theorem}\label{suff for nec} Suppose $\bar x \in \Gamma \cap \Theta $ with $\bar x \in \partial \Theta$ is local optimal for \eqref{eq.problem2}. 
\begin{enumerate}
\item[(a)]  If  the  EMFCQ   
  holds at  $\bar x$,    then there exist $y_j,\ j\in\Ecal$ and $y_j \ge 0,\ j\in \{0\}\cup\bar \Acal$, such that 
 \begin{equation}\label{nec nec}  \nabla_i f_0(\bar x) +  y_0p\sign(\bar x_i) |\bar x_i|^{p-1}  +  \sum_{ j \in\bar\Acal\cup\Ecal } y_j \nabla_i f_j(\bar x) = 0, i\in\bar\Ncal.    
 \end{equation} 
\item[(b)]  Suppose $\Gamma$ is closed and convex.  If $v=0$ is the only vector such that 
$v\in N_{\Theta}(\bar x)$ and $-v\in N_\Gamma(\bar x)$, then $N_{\Theta\cap\Gamma} (\bar x)= N_\Theta(\bar x) + N_\Gamma(\bar x)$. Therefore, if $\bar x$ is local optimal for \eqref{eq.problem2}, then 
\[ - \nabla_i f(\bar x) + y_0 p\sign(\bar x_i) |\bar x_i|^{p-1} + v_i = 0, i\in \bar\Ncal;  \  y_0\ge 0; \  v\in N_\Gamma(\bar x). \]

\item[(c)]  Suppose $\Gamma = \mathbb{R}^n$.  If $\bar x$ is local optimal for \eqref{eq.problem2}, then there exists $y_0\geq 0$ such that
\[ - \nabla_i f(\bar x) + y_0 p\sign(\bar x_i) |\bar x_i|^{p-1}  = 0, i\in \bar\Ncal. \]
\end{enumerate} 
\end{theorem}
\begin{proof}
(a)  From  \cite[Theorem 6.14]{rockafellar2009variational}, 
 if the  EMFCQ holds for \eqref{eq.problem2} at  $\bar x \in \Gamma \cap \Theta $ with $\bar x \in \partial \Theta$,  
then $\Theta\cap\Gamma$ is regular at $\bar x$ and  
\begin{equation} 
\widehat N_{\Theta\cap\Gamma}(\bar x) =  \Upsilon_{\Theta\cap\Gamma}(\bar x)^*.
\end{equation} 
By \Cref{basic.thm}, (a) is true. 
 
(b) By \cite[Theorem 6.42]{rockafellar2009variational}, $N_{\Theta\cap\Gamma}(\bar x) = N_\Theta(\bar x) + N_\Gamma(\bar x)$. Therefore, if $\bar x$ is 
local optimal, then $\bar x\in \widehat N_{\Theta\cap\Gamma} (\bar x) \subset  N_{\Theta\cap\Gamma} (\bar x)  = N_\Theta(\bar x) + N_\Gamma(\bar x)$.

(c) Trivial by (b). 
\end{proof}

 We call the conditions \eqref{nec nec} the   Karush-Kuhn-Tucker (KKT) conditions for \eqref{eq.problem2}.  Using the notation of $\Lambda_\Gamma$, 
 it can also be equivalently written as 
 \begin{equation}\label{kkt.2.lambda}
 - \nabla f_0(\bar x) \in N_{\Theta}(\bar x) +  \Lambda_{\Gamma}(\bar x).
 \end{equation}

\section{First-order sequential optimality condition}
In this section, we study the sequential optimality conditions under the approximate  Karush-Kuhn-Tucker (AKKT) conditions, 
which 
are defined as follows.

\begin{definition}
\begin{enumerate}
\item[(i)] 	For \eqref{eq.problem1}, we say that $\bar x\in \Gamma$ satisfies  the   AKKT  if there exist   $\{x^\nu\}\subset \mathbb{R}^{n}$, 
	$\{ y_j^\nu\}\subset\mathbb{R}$, $j\in \Ecal\cup\bar\Acal$    such that $\lim\limits_{\nu\to\infty}x ^\nu=\bar x$,  
	 $y_j^\nu \ge 0, j \in  \bar\Acal$  and  
	\begin{equation*}\label{KKT 1}
	\lim\limits_{\nu\to\infty}\nabla_i f(x ^\nu) +  \lambda   p  \sign(x^\nu_i) |x^\nu_i|^{p-1}+\sum\limits_{j\in\Ecal\cup\bar\Acal}y_j^\nu \nabla_i f_j(x ^\nu)   =   0,\  \forall i\in  \bar\Ncal.    
          \end{equation*}
\item[(ii)] For \eqref{eq.problem2}, we say that $\bar x\in \Theta\cap\Gamma$ with $\bar x\in\partial \Theta$ satisfies  the    AKKT  if there exist   $\{x^\nu\}\subset \mathbb{R}^{n}$, 
	$\{ y_j^\nu\}\subset\mathbb{R}$, $j\in \{0\}\cup \Ecal\cup\bar\Acal$    such that $\lim\limits_{\nu\to\infty}x ^\nu=\bar x$,  
	 $y_j^\nu \ge 0, j \in \{0\}\cup \bar\Acal$  and  
	\begin{equation*}\label{KKT 1}
	\lim\limits_{\nu\to\infty}\nabla_i f(x ^\nu) +  y_0^\nu   p  \sign(x^\nu_i) |x^\nu_i|^{p-1}+\sum\limits_{j\in\Ecal\cup\bar\Acal}y_j^\nu \nabla_i f_j(x ^\nu)   =   0,\  \forall i\in  \bar\Ncal.    
          \end{equation*}
         \end{enumerate}
\end{definition}

Next we provide properties  of \eqref{eq.problem1} and \eqref{eq.problem2} under which the AKKT implies the  KKT. This property is named the extended cone-continuity property (ECCP), which is defined as follows. 

\begin{definition}
\begin{enumerate}
\item[(a)] We say that $\bar x\in \Gamma$ satisfies the  ECCP  for \eqref{eq.problem1} 
 if the set-valued mapping $\lambda\partial \phi(x) + \Lambda_\Gamma(x)$ is outer semicontinuous at $\bar x$, that is, 
\[ \limsup_{x^\nu \to\bar x} [\lambda\partial \phi(x^\nu) + \Lambda_\Gamma(x^\nu)] \subset \lambda\partial \phi(\bar x) + \Lambda_\Gamma(\bar x). \] 
If $\Gamma$ is a closed and convex set,   $\Lambda_\Gamma$ can be replaced by $N_\Gamma$. 
\item[(b)] Similarly,  we say $\bar x\in \Theta \cap \Gamma$ satisfies the  ECCP  for \eqref{eq.problem2} 
 if the set-valued mapping $\Lambda_{\Theta\cap\Gamma}(x)$ is outer semicontinuous at $\bar x$, that is, 
\[ \limsup_{x^\nu \to\bar x} [ N_{\Theta}(  x^\nu) +  \Lambda_{\Gamma}(  x^\nu)] \subset N_{\Theta}(\bar x) +  \Lambda_{\Gamma}(\bar x). \] 
If $\Gamma$ is a closed and convex set,   $\Lambda_\Gamma$ can be replaced by $N_\Gamma$. 
\end{enumerate}
\end{definition}

In the following theorem we show that if ECCP holds,  the  AKKT implies  the KKT. 
\begin{theorem}\label{Px AKKT } 
\begin{enumerate}
\item[(a)] 
For \eqref{eq.problem1}, suppose the AKKT holds at $\bar x \in \Gamma$. If the ECCP holds at $\bar x$, then the KKT  condition  \eqref{kkt.1.lambda}  are satisfied  at $\bar x$. 
\item[(b)] 
For \eqref{eq.problem2}, suppose the AKKT holds at $\bar x \in \partial\Theta\cap\Gamma$. If the ECCP holds at $\bar x$, then the KKT condition \eqref{kkt.2.lambda}   are satisfied  at $\bar x$.
\end{enumerate}
\end{theorem}
\begin{proof} (a)    
First of all,    $\bar\Ncal \subset \Ncal(x ^\nu)$   for $ x^\nu$ sufficiently close to $\bar x$.  Since the   AKKT  condition holds at $\bar x$, there 
exist   $\{x ^\nu\}\to \bar x$ and $\{w^\nu\} \subset \mathbb{R}^n$   such that 
$ \nabla  f(x ^\nu) +  {w}^\nu  \to 0$,  
where  
$ {w}^\nu \in     \lambda\partial\phi(  x^\nu) +  \Lambda_\Gamma( x^\nu)$. 
	Taking limits and using the continuity of the gradient of $f$  near $\bar x$,  we obtain
	\begin{equation*}
		-\nabla f(\bar x) = \lim_{\nu\to\infty}  {w}^\nu \in \limsup_{\nu\to\infty}  [\lambda\partial\phi(  x^\nu) +  \Lambda_\Gamma( x^\nu) ]\subset\limsup_{x \to \bar x} [\lambda\partial\phi(  x ) +  \Lambda_\Gamma( x )]\subset \lambda\partial\phi(\bar x) + \Lambda_\Gamma(\bar x),
	\end{equation*}
where the last inclusion follows from the ECCP. Therefore,  $-\nabla  f(\bar x) \in \lambda \partial\phi(\bar x) + \Lambda_\Gamma(\bar x)$. 

(b) Similarly,  since the   AKKT  condition holds at $\bar x$, there 
exist   $\{x ^\nu\}\to \bar x$ and $\{w^\nu\} \subset \mathbb{R}^n$   such that 
$\nabla  f(x ^\nu)  +   w^\nu  \to 0$,
where $ w^\nu  \in N_{\Theta}(  x^\nu) +  \Lambda_{\Gamma} (  x^\nu)$. 
	Taking limits and using the continuity of the gradient of $f$   near $\bar x$,  we obtain
	\begin{equation*}
		-\nabla  f(\bar x)  = \lim_{\nu\to\infty}  {w}^\nu \in \limsup_{\nu\to\infty} N_{\Theta}(  x^\nu) +  \Lambda_{\Gamma}(  x^\nu)  \subset   
		\lim_{x\to\bar x}  N_{\Theta}(  x) +  \Lambda_{\Gamma}(  x)   \subset  N_{\Theta}(\bar x) +  \Lambda_{\Gamma}(\bar x),
	\end{equation*}
where the last inclusion follows from the ECCP. Therefore,  $-\nabla f(\bar x) \in N_{\Theta}(\bar x) +  \Lambda_{\Gamma}(\bar x)$. 
\end{proof}

We discuss the cases when the ECCP holds true.  For \eqref{eq.problem1}, we have the following results. 
 
 \begin{proposition}\label{prop.ccp1} 
 For \eqref{eq.problem1},  the ECCP holds true for any  of the following cases.  
 \begin{enumerate}  
  \item[(a)]     $\Lambda_\Gamma$ is  outer semicontinuous at $\bar x$ and  $\partial^\infty \phi(\bar x)$ contains no vector $v\ne 0$ such that 
$-v\in  \Lambda_\Gamma(\bar x)$.     
\item[(b)] The EMFCQ is satisfied at $\bar x$. 
\item[(c)]    $\Gamma$ is a closed and convex set  and  $\partial^\infty \phi(\bar x)$ contains no vector $v\ne 0$ such that 
$-v\in  N_ \Gamma(\bar x)$.   
 \end{enumerate}
   \end{proposition}
   
   \begin{proof} (a) 
    Let $\bar w$ be an element of $\limsup\limits_{x\to \bar x}  [\lambda\partial\phi(  x) + \Lambda_\Gamma(x)]$, 
   so there are sequences $\{x^\nu\}$, $\{w^\nu\}$, $\{u^\nu\}$ such that 
   $x^\nu \to \bar x$, $w^\nu \to \bar w$ and 
  \begin{equation}\label{eq.w}  w^\nu_i = u_i^\nu + \sum_{j\in\Ecal\cup\bar\Acal} y_j^\nu \nabla_i f_j(x^\nu)
  \end{equation} 
   with  $u^\nu \in \partial \lambda\phi(  x^\nu) $ and $y_j^\nu \in \mathbb{R}_+, j\in\bar\Acal$. 
   
    If $\{y_j^\nu\}, j\in\Ecal\cup\bar\Acal$ are bounded, then they all have limits $\bar y_j, j\in\Ecal\cup\bar\Acal$; moreover, $\{u^\nu\}$ is also bounded and 
     $\bar u := \lim_{\nu\to\infty} u^\nu \in \lambda\partial\phi(\bar x)$ (possibly taking limits on a convergent subsequence)
       due to the outer semicontinuity of $\partial\phi$.  
    By possibly extracting an convergent 
    subsequence,  we have 
  \[ u^\nu +  \sum_{j\in\Ecal\cup\bar\Acal} y_j^\nu \nabla_i f_j(x^\nu) \to \bar u +  \sum_{j\in\Ecal\cup\bar\Acal} \bar y_j \nabla_i f_j(\bar x) \in \lambda\partial\phi(\bar x)+\Lambda_\Gamma(\bar x),\]  
  meaning $\lambda \partial \phi  + \Lambda_\Gamma$ is outer semicontinuous at $\bar x$. 
 
      If $\{y_j^\nu\}, j\in\Ecal\cup\bar\Acal$ are unbounded,  letting $M^\nu = \max\{ |y_j^\nu|,  j\in\Ecal\cup\bar\Acal\} $.  
      Dividing \eqref{eq.w} by $M^\nu$, we arrive at 
        \[\frac{w^\nu}{M^\nu} =  \frac{u^\nu}{M^\nu} +  \sum_{j\in\Ecal\cup\bar\Acal} \frac{y_j^\nu}{M^\nu} \nabla_i f_j(x^\nu)\] 
 Since $\max\{ y_j^\nu/M^\nu, j\in\{0\}\cup\Ecal\cup\bar\Acal\}=1$ for all $\nu$, they have nonzero limit point $\tilde y_j, j\in\Ecal\cup\bar\Acal$. 
 Moreover,   
 $\bar u := \lim_{\nu\to\infty} \frac{u^\nu}{M^\nu} \in \lambda \partial^\infty\phi(\bar x)$,  we can extract a convergent 
   subsequence. Thus, taking limits above, we get  
   \[   \lambda \partial^\infty\phi(\bar x)  \ni \bar u = -    \sum_{j\in\Ecal\cup\bar\Acal} \tilde y_j  \nabla_i f_j(\bar x )  \in \Lambda_\Gamma(\bar x),\]      a contradiction.

(b) If the EMFCQ holds at $\bar x$,  then  $\partial^\infty \phi(\bar x)$ contains no vector $v\ne 0$ such that 
$-v\in  \Lambda_\Gamma(\bar x)$ and $\Lambda_\Gamma = N_\Gamma$ by \Cref{optimality.p1}(a). This reverts to (a). 
 
(c) Let $\bar w$ be an element of $\limsup\limits_{x\to \bar x}  [\lambda\partial\phi(  x) + N_\Gamma(x)]$, 
   so there are sequences $\{x^\nu\}$, $\{w^\nu\}$, $\{u^\nu\}$, $\{v^\nu\}$ such that 
   $x^\nu \to \bar x$, $w^\nu \to \bar w$ and 
    \begin{equation}\label{eq.w2} 
     w^\nu= u^\nu + v^\nu
     \end{equation} 
       with  $u^\nu \in \partial \lambda\phi(  x^\nu) $ and $v^\nu\in N_\Gamma(x^\nu)$. 
   
    If $\{v^\nu\}$ are bounded, then $\{u^\nu\}$ and $\{v^\nu\}$ all have limits $\bar u$ and $\bar v$.  Moreover,   $\bar u := \lim\limits_{\nu\to\infty} u^\nu \in \lambda\partial\phi(\bar x)$ 
    and $\bar v\in N_\Gamma(\bar x)$  
     (possibly taking limits on a convergent subsequence)
       due to the outer semicontinuity of $\partial\phi$ and $N_\Gamma$.  
    By possibly extracting an convergent 
    subsequence,  we have 
  \[ u^\nu + v^\nu \to \bar u +  \bar v \in \lambda\partial\phi(\bar x)+\Lambda_\Gamma(\bar x),\]  
  meaning $\lambda \partial \phi  + \Lambda_\Gamma$ is outer semicontinuous at $\bar x$. 
 
      If $\{v^\nu\}$ are unbounded,  letting $M^\nu =  \|v^\nu\|$.  
      Dividing \eqref{eq.w} by $M^\nu$, we arrive at 
        \[\frac{w^\nu}{M^\nu} =  \frac{u^\nu}{M^\nu} +    \frac{v^\nu}{M^\nu}.\] 
 Since  $\frac{v^\nu}{M^\nu}=1$ for all $\nu$, it has nonzero limit point $\tilde v\in N_\Gamma(\bar x)$ due to the outer semicontinuity of $N_\Gamma$. 
 Moreover,   
 $\bar u := \lim\limits_{\nu\to\infty} \frac{u^\nu}{M^\nu} \in \lambda \partial^\infty\phi(\bar x)$,  we can extract a convergent 
   subsequence. Thus, taking limits above, we get  
   \[   \lambda \partial^\infty\phi(\bar x)  \ni \bar u = -  \bar v \in N_\Gamma(\bar x),\]      a contradiction.  
   \end{proof}

As for \eqref{eq.problem2}, we have the following results. 
 
 \begin{proposition}\label{prop.ccp2}
   For \eqref{eq.problem2},  the ECCP holds true for any of the following cases. 
 \begin{enumerate}  
\item[(a)]      $\Lambda_\Gamma$ is  outer semicontinuous at $\bar x$ and the only solution for $v_1+v_2 = 0$ with $v_1\in N_\Theta(\bar x)$ and $v_2\in \Lambda_\Gamma(\bar x)$ 
is $v_1=v_2=0$.    
\item[(b)] The  EMFCQ is satisfied at $\bar x$.      
\item[(c)]    $\Gamma$ is a closed and convex set  and  and the only solution for $v_1+v_2 = 0$ with $v_1\in N_\Theta(\bar x)$ and $v_2\in N_\Gamma(\bar x)$ 
is $v_1=v_2=0$.      
\end{enumerate}
   \end{proposition}

\begin{proof} (a) The proof is similar to the argument for \Cref{prop.ccp1}(a) by replacing the role of  $u^\nu\in\lambda\partial\phi(\bar x)$ with  $y_0 p\text{sign}(x_{\bar\Ncal}^\nu)|x_{\bar\Ncal}^\nu|^{p-1}\in N_\Theta(x^\nu)$ ($x^\nu$ 
is selected sufficiently close to $\bar x$ so that $\bar\Ncal \subset \Ncal(x^\nu)$) and considering the boundedness of $\{y_j^\nu\}, j\in\{0\}\cup \Ecal\cup\bar\Acal$ instead of 
$\{y_j^\nu\}, j\in\Ecal\cup\bar\Acal$.  Therefore, we skip the details of the proof. 

(b) The EMFCQ   for \eqref{eq.problem2} is equivalent to saying that the only solution for $v_1+v_2 = 0$ with $v_1\in N_\Theta(\bar x)$ and $v_2\in \Lambda_\Gamma(\bar x)$ 
is $v_1=v_2=0$.   Moreover, notice that the EMFCQ   for \eqref{eq.problem2} implies the EMFCQ for \eqref{eq.problem1}. Therefore,   we have from   \cite[Theorem 6.14]{rockafellar2009variational} 
 that $\Gamma$ is regular at $\bar x$ and $N_\Gamma(\bar x) = \Lambda_\Gamma(\bar x)$. This case then reverts to (a). 
 
 (c)  The proof is similar to the argument for \eqref{prop.ccp1}(c) by replacing the role of  $\lambda\partial\phi$ with  $N_\Theta$; therefore it is skipped. 
\end{proof}

\section{Convergence analysis using AKKT} 

The research on algorithms for solving general constrained problems involving  $\ell_p$ norms is so far limited. 
 To the best of our knowledge, only simple cases such as linearly constrained or convex set constrained cases have been studied. 
 When solving for general constrained problems with $\ell_p$ norm,  many works focus on the reformulations where the  
 constraint violation is penalized in the objective \cite{lu2014iterative, chartrand2008iteratively}.   
 We now extend the existing algorithms for solving unconstrained cases to general constrained cases  and prove the global convergence 
 by showing that  AKKT  conditions discussed in the previous 
 section are satisfied at the limit point.

We extend the framework of iteratively reweighted $\ell_1$ methods for solving unconstrained $\ell_p$ regularization problems \cite{2019Relating} to 
constrained cases \eqref{eq.problem1},   
\begin{equation}
	\min\ F(x ) := f_0(x )+ \lambda\|x \|^p_p \quad \text{ s.t. }\  f_j(x )\leq 0, \ \forall j\in\mathcal{I}; \  f_j(x )=0,\ \forall j\in\mathcal{E}. 
\end{equation}
where $f:\mathbb{R}^n\to\mathbb{R}$ is   continuously Lipschitz differentiable   with constant $L_f\geq 0$.  Here 
$\Gamma$ is assumed to be closed and  convex.

We first formulated a smooth approximation $F(x ; {\epsilon})$ of $F(x )$
$$F(x ;  {\epsilon}):=f_0(x )+\lambda\sum\limits_{i=1}^n (|x_i|+\epsilon_i)^p,$$
where $\epsilon\in\mathbb{R}^n_{++}$.
At $k$th iteration,   a convex local model to approximate $F(x; {\epsilon})$ is constructed 
$$G(x ;x ^k, {\epsilon}^k):= Q_k(x )+\lambda \sum\limits_{i=1}^n w(x^k_i,\epsilon_i^k)|x_i|$$
where the weights are given by $w(x^k_i,\epsilon_i^k)=p(|x^k_i|+\epsilon^k_i)^{p-1}$ and $Q_k(x )$ represents a local approximation mode to $f$ at $x ^k$ and is generally assumed to be smooth and convex. 

\begin{algorithm}[h] 
	\caption{General framework of iteratively reweighted $\ell_1$ approach}  
   \label{alg::irl1}  
   \begin{algorithmic}[1]  
	 \Require
		 $\alpha\in(0,1)$, $ {\epsilon}^0\in\mathbb{R}^n_{++}$ and $x ^0$.
	 \State Initialization: Set $k=0$.
	 \Repeat
		 \State (Reweighing) Compute $w^k_i={p}({|x^k_i|+\epsilon}^k_i)^{p-1}$, $i= 1,\cdots,n$. 
		   \State (Solving the Subproblem) $x ^{k+1}\leftarrow \mathop{\arg\min}\limits_{x \in\Gamma}\{Q_k(x )+\lambda \sum\limits_{i=1}^n w^k_i |x_i|\}$.  
		   
	   \State (Update $\epsilon$)  Set $ {\epsilon}^{k+1}\in(0,\alpha {\epsilon}^k )$.
	   \State Set $k\leftarrow k+1$ and go to step~$3$.
	   \Until{Convergence}
   \end{algorithmic}  
 \end{algorithm} 
As in  \cite{2019Relating}, we make the following assumptions about the choice of $(x ^0, {\epsilon}^0)$ and $Q_k(\cdot)$. 
\begin{assumption}\label{assumption tmp}
	The initial point $(x ^0, {\epsilon}^0)$ and local   model $Q_k(\cdot)$ are such that \\
	(i) The level set $\mathcal{L}(F^{0}):=\{x  \mid F(x ) \leq F^{0}\}$ is bounded where $F^{0}:=F(x ^{0};  {\epsilon}^{0})$.\\
	(ii) For all $k \in \mathbb{N}, \nabla Q_{k}(x ^k )=\nabla f(x ^k )$, $Q_{k}(\cdot)$ is strongly convex with constant $M>L_{f} / 2>$ 0.
	\end{assumption}

First, we show that $F(x , {\epsilon})$ is monotonically decreasing over the iterates $(x ^k, {\epsilon}^k)$. Define the following two terms
$$
\begin{aligned}
\Delta F(x ^{k+1},  {\epsilon}^{k+1}) &:=F(x ^k ,  {\epsilon}^k )-F(x ^{k+1},  {\epsilon}^{k+1}) \\
\Delta G(x ^{k+1} ; x ^k ,  {\epsilon}^k ) &:=G(x ^k  ; x ^k ,  {\epsilon}^k )-G(x ^{k+1} ; x ^k ,  {\epsilon}^k ) ,
\end{aligned}
$$
and use the shorthands $ W^k :=\operatorname{diag}\left(w_{1}^k , \ldots, w_{n}^k \right)$.
\begin{proposition}\label{prop.x-x}
	Suppose Assumption \ref{assumption tmp} holds. Let $\{(x ^k, {\epsilon}^k)\}$ be the sequence generated by Algorithm \ref{alg::irl1} . It follows that $F(x , {\epsilon})$ is monotonically decreasing over $\{(x ^k, {\epsilon}^k)\}$ and  
	\begin{equation}\label{monotonically}
		(M-\frac{L_f}{2})\sum\limits_{t=0}^{k-1}\|x ^{t+1}-x ^t\|^2_2 \le  F(x ^0, {\epsilon}^0)-F(x ^k, {\epsilon}^k). 
	\end{equation}
	Hence, $\lim\limits_{k\to\infty}\|x ^{k+1}-x ^k\|_2=0.$ 
\end{proposition}
\begin{proof} From the same argument as the proof for \cite[Proposition 1]{2019Relating}, we have 
\begin{equation}\label{monotonically 5}
	\Delta F(x ^{k+1},  {\epsilon}^{k+1}) \geq \Delta G(x ^{k+1} ; x ^k ,  {\epsilon}^k )+\frac{M-L_{f}}{2}\|x ^k -x ^{k+1}\|_{2}^{2}.
\end{equation}
Assumption \ref{assumption tmp} implies the subproblem solution $x ^{k+1}$ satisfies the optimality condition 
\begin{equation}\label{monotonically 3}
	\nabla Q_k(x ^{k+1})+\lambda W^k {y}^{k+1}+ {z}^{k+1}=0
\end{equation}
with $ {z}^{k+1}\in N_{\Gamma}(x ^{k+1})$ and $ {y}^{k+1}\in\partial \|x ^{k+1}\|_1$. Hence,
\begin{equation}\label{monotonically 4}
	\begin{aligned}
		&\ \Delta G(x ^{k+1} ; x ^k ,  {\epsilon}^k ) \\
		=&\ Q_{k}(x ^k )-Q_{k}(x ^{k+1})+\lambda \sum_{i=1}^{n} w_{i}^k (|x_{i}^k |-|x_{i}^{k+1}|) \\
		\geq &\ \langle\nabla Q_{k}(x ^{k+1}),x ^k -x ^{k+1}\rangle +\frac{M}{2}\|x ^{k+1}-x ^k \|_{2}^{2}+\lambda \sum_{i=1}^{n} w_{i}^k  y_{i}^{k+1}(x_{i}^k -x_{i}^{k+1})+z^{k+1}_i(x_{i}^k -x_{i}^{k+1}) \\
		=&\ \langle\nabla Q_{k}(x ^{k+1})+\lambda W^k   {y}^{k+1}+ {z^{k+1}},x ^k -x ^{k+1}\rangle+\frac{M}{2}\|x ^{k+1}-x ^k \|_{2}^{2} \\
		=&\ \frac{M}{2}\|x ^{k+1}-x ^k \|_{2}^{2},
		\end{aligned}
\end{equation}
where the inequality is by Assumption \ref{assumption tmp}, the convexity of $|\cdot|$ and the definition of normal cone, and the last equality is by \eqref{monotonically 3}.  
Combining \eqref{monotonically 5} and \eqref{monotonically 4}, we have 
\begin{equation}
	\Delta F(x ^{k+1},  {\epsilon}^{k+1}) \geq (M-\frac{L_{f}}{2})\|x ^k -x ^{k+1}\|_{2}^{2}.
\end{equation}
Replacing $k$ with $t$ and summing up from $t=0$ to $k-1$, we have 
$$\sum\limits_{t=0}^{k-1}(F(x ^t, {\epsilon}^t)-F(x ^{t+1}, {\epsilon}^{t+1}))\geq (M-\frac{L_{f}}{2})\sum\limits_{t=0}^{k-1}\|x ^t-x ^{t+1}\|_2^2,$$
completing the proof of \eqref{monotonically}.  
\end{proof}

Now we prove that every cluster point $\bar x$ of  $\{x ^k\}$ generated by Algorithm \ref{alg::irl1}  satisfied the AKKT. As a result, it satisfies the    first-order necessary optimality condition of \eqref{eq.problem1}  by \Cref{Px AKKT } if the CCP condition holds at $\bar x$. 
\begin{theorem} 
Suppose $\{x ^k\}$ is the sequence generated by Algorithm \ref{alg::irl1}  for solving \eqref{eq.problem1}.  It holds that every cluster point 
of $\{ x^k\}$  satisfies the AKKT  condition.
\end{theorem}
\begin{proof}  Let $\bar x$ be a cluster point of $\{ x^k\}$ with subsequence $\{ x^k\}_S \to \bar x$.  
By \Cref{prop.x-x}, $\{ x^{k+1}\}_S \to \bar x$.   From the optimality 
condition of the subproblem of   Algorithm \ref{alg::irl1} , we have 
	\begin{equation}
	\nabla_i f(x ^k) +M(x^{k+1}_i-x^k_i)+\lambda p(|x^k_i|+\epsilon^k_i)^{p-1}\sign(x^{k+1}_i)+z^{k+1}_i=0,\quad \forall i\in\Ncal^{k+1},
	\end{equation}
	where $ {z}^{k+1}\in N_{\Gamma}(x ^{k+1})$.
	For any $i \in \Ncal^{k+1}$,
	\begin{equation}\label{convergence tmp}
		\begin{aligned}
		&\ | \nabla_i f(x ^{k+1}) +\lambda p|x^{k+1}_i|^{p-1}\sign(x^{k+1}_i)+z^{k+1}_i|\\
		=&\ |(\nabla_i f(x ^{k+1})-\nabla_i f(x ^{k }))+\lambda p\sign(x^{k+1}_i)(|x^{k+1}_i|^{p-1}-(|x^k_i|+\epsilon^k_i)^{p-1})-M(x^{k+1}_i-x^k_i)|\\
		\leq &\ (L_f-M)|x^{k+1}_i-x^k_i|+\lambda p(1-p)|\hat x^k _i|^{p-2}||x^{k+1}_i|-|x^k_i|-\epsilon^k_i|\\
 		\end{aligned}
	\end{equation}
	where $\hat x^k _i$ is between $|x^{k+1}_i|^{p-1}$ and $(|x^k_i|+\epsilon^k_i)^{p-1}$. 
	For sufficiently large $k$,  $|\hat x^k _i|^{p-2} > \delta > 0$,  $\Ncal(\bar x) \subset \Ncal(x^k)$ and $\Ncal(\bar x)\subset \Ncal(x^{k+1})$. 
	Therefore, for any $i\in\Ncal(\bar x)$, 
\begin{align*}
       & \  | \nabla_i f(x ^{k+1}) +\lambda p|x^{k+1}_i|^{p-1}\sign(x^{k+1}_i)+z^{k+1}_i|\\
\leq & \   (L_f-M)|x^{k+1}_i-x^k_i|+\lambda p(1-p)\delta^{p-2}|x^{k+1}_i-x^k_i| +\lambda p(1-p)\delta^{p-2} \epsilon^k_i  
 \to   0
 \end{align*} 
			as $k\to\infty$. 
This completes the proof.
\end{proof}

\section{Numerical experiment}

In this section, we design numerical experiments to demonstrate the   performance of  Algorithm \ref{alg::irl1}   for solving the   \eqref{eq.problem1}  with  application   in portfolio management, where    $f$ can include various  loss functions such as  the variance of portfolio or tracking error of index tracking. Here we use the   commonly seen Markowitz mean-variance model to predict an optimal portfolio. Specifically, we only consider the shorting-prohibited Markowitz model and assume the optimal Lagrangian multiplier associated with the mean constraint is known as $\eta$ which can be chosen accordingly to   a reasonable expected return rate.

In the experiment, we collected historical daily stock price data to obtain $R $ and $m$ in $S\&P \ 500$ index from  Yahoo finance\footnote{\url{finance.yahoo.com}}, which spans from $01/01/2013$ to $31/ 12/2013$. We do not include any company unless it is traded on the market at least $90\%$ of the trading days during the data period, nor do any company not listed on the market for the entire timescale. The total list has 471 companies by 251 trading days. We choose $\eta = 0.001$ and $\lambda = 0.001$, and recast the Markowitz model with no-shorting constraint as a linear equality constrained optimization problem with $\ell_{\frac{1}{2}}$ regularizer,
\begin{equation}\label{portfolio}
\begin{aligned}
\min\quad  &\frac{1}{2} x^{T} R  x- \eta \mu^{T} x +\lambda \|x\|^{\frac{1}{2}}_{\frac{1}{2}}\\
\text { s.t.}\quad & e^{T} x=1, \\
& x \geq 0
\end{aligned}
\end{equation}
where $R \in \mathbb{R}^{n\times n}$ is the estimated covariance matrix of the portfolio, and $  \eta \mu$ is  the   estimated return vector  with $\eta > 0$ and $\mu\in\mathbb{R}^n$.

In \Cref{alg::irl1}, we use $Q_k(x) = (R x^k  - \eta \mu)^Tx+\frac{\beta}{2}\|x-x^k\|_2^2$.   The parameters are selected as 
$\alpha=0.998$, $\beta=1.1 L_f$, $ {\epsilon}^0=0.001$ and the  initial point is  $x ^0=\frac{1}{n}e$ so that it is feasible. 
 As mentioned in the last section, AKKT conditions   are satisfied at the clustering point and the ECCP holds true by \Cref{prop.ccp1}(c) since $\Gamma$ here is a closed and convex set. Therefore, from \Cref{Px AKKT }(a), the KKT condition of \eqref{portfolio} is satisfied at the clustering point generated by the algorithm.

After solving each subproblem, we obtain a primal feasible iterate $x^{k+1}$ and dual iterate $\nu^k$ satisfying 
the optimality condition for the subproblem 
\begin{equation}
	(R_i x^k-\eta\mu)^Tx^{k+1} +\beta(x^{k+1}_i-x^{k}_i)+\lambda w^k_i+\nu^k=0,\ i\in\Ncal(x^{k+1}), 
\end{equation}
where $R_i$ is the $i$th row of $R$. 
For any primal-dual pair $(x,\nu)$ with $x\geq 0$, we can define the following metric to measure the optimality residual at $(x,\nu)$ 
  \begin{equation}
	  \alpha(x,\nu):=\sum\limits_{i\in\mathcal{N}(x)}|x^TR _i-c_i+\lambda p x_i^{p-1}+\nu|, 
  \end{equation}
  where $\nu$ is the dual variable associated with  the simplex constraint.

\begin{figure}[htbp]
	\centering
	\includegraphics[width=0.7\linewidth]{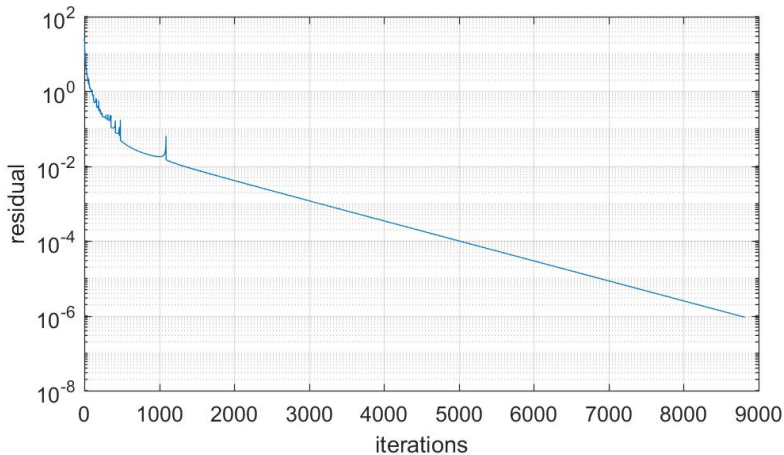}
	\caption{The residual $\alpha(x^k,\nu^k)$ generated by \Cref{alg::irl1} } \label{figure:residual}
  \end{figure}

We plot the evolution of $\alpha(x^k, \nu^k)$ over iterations in 
\Cref{figure:residual}, which steadily decreases to 0.  Figure \ref{sparsity} shows the number of nonzero components of the portfolio versus the regularization parameter $\lambda$ by fixing $\eta=0.001$.  
We can see that larger $\lambda$ can yield   sparser  solutions. 
\begin{figure}[htbp]
	\centering
	\includegraphics[width=0.8\textwidth]{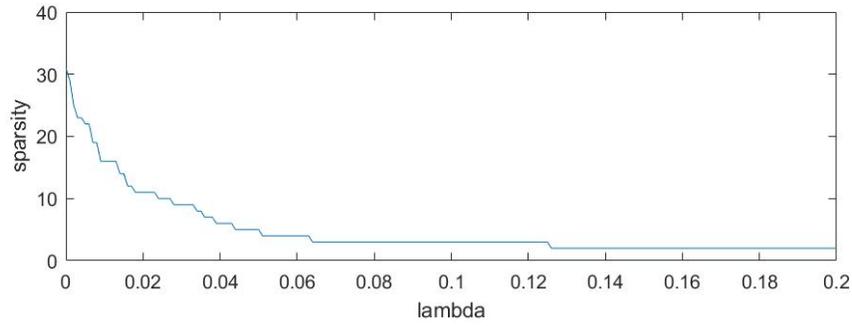}
	\caption{Portfolio sparsity (nonzero components) for different $\lambda$}\label{sparsity}	
\end{figure}

For out-of-sample testing, we collected historical daily stock price data in $S\&P \ 500$ index from  Yahoo finance, which spans from $01/01/2014$ to $31/03/2014$. 
Figure \ref{sharpe} shows the Sharpe ratios of our $\ell_p$-norm regularized portfolio. The sparse portfolios are more implementable due to the transaction costs or physical limitations reasons. Our results indicate that an intermediate sparsity (around 10) portfolio may have the best   Sharpe ratio  in both in-sample or out-of-sample performance.   
\begin{figure}[htbp]
	\centering
	\includegraphics[width=0.6\textwidth]{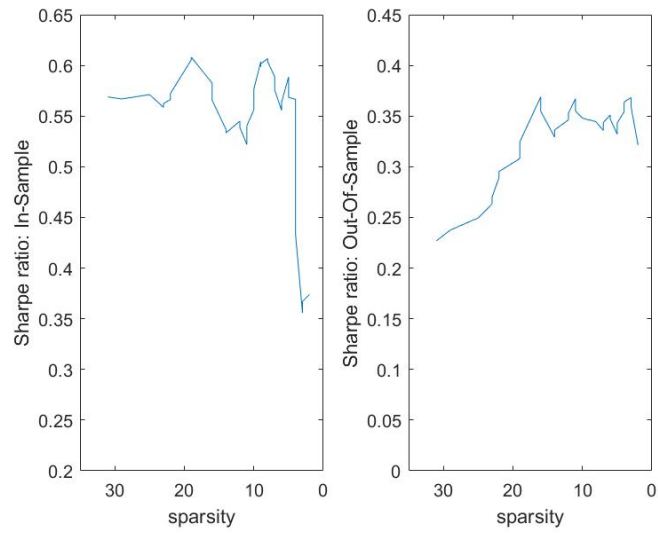}
	\caption{Portfolio Sharpe ratios for different sparsity: in-sample and out-of-sample}\label{sharpe}	
\end{figure}

\bibliography{references}
\end{document}

%% file: style_for_notes.tex
\usepackage[letterpaper, margin=1in]{geometry} 
 
\usepackage{mathrsfs,amsmath,amsfonts,amssymb,latexsym,amsthm,color,
                         graphicx,graphics,float,subfigure,epsf,enumerate,hyperref,algorithm,multirow}
\usepackage{algpseudocode,cleveref}  

\newtheorem{theorem}{Theorem}[section]
 \newtheorem{corollary}[theorem]{Corollary}
 \newtheorem{lemma}[theorem]{Lemma}
  \newtheorem{assumption}[theorem]{Assumption}
 \newtheorem{proposition}[theorem]{Proposition}

 \theoremstyle{Definition}
 \newtheorem{definition}[theorem]{Definition}
 \newtheorem{example}[theorem]{Example}
 \theoremstyle{remark}
 \newtheorem{remark}[theorem]{Remark}
   \newtheorem{exercise}{Exercise}
 \numberwithin{equation}{section}

\newcommand{\st}      {{\rm{s.t.}}}

\newcommand{\sign}{{\rm{sign}}}

\newcommand{\back}{\begin{acknowledgements}} \newcommand{\eack}{\end{acknowledgements}}
\newcommand{\balg}{\begin{algorithm}}    \newcommand{\ealg}{\end{algorithm}}
\newcommand{\balc}{\begin{algorithmic}}  \newcommand{\ealc}{\end{algorithmic}}
\newcommand{\bali}{\begin{aligned}}      \newcommand{\eali}{\end{aligned}}
\newcommand{\barr}{\begin{array}}        \newcommand{\earr}{\end{array}}
\newcommand{\bass}{\begin{assumption}}   \newcommand{\eass}{\end{assumption}}
\newcommand{\bbma}{\begin{bmatrix}}      \newcommand{\ebma}{\end{bmatrix}}
\newcommand{\bcas}{\begin{cases}}        \newcommand{\ecas}{\end{cases}}
\newcommand{\bcen}{\begin{center}}       \newcommand{\ecen}{\end{center}}
\newcommand{\bcol}{\begin{column}}       \newcommand{\ecol}{\end{column}}
\newcommand{\bcos}{\begin{columns}}      \newcommand{\ecos}{\end{columns}}
\newcommand{\bcon}{\begin{condition}}    \newcommand{\econ}{\end{condition}}
\newcommand{\bcor}{\begin{corollary}}    \newcommand{\ecor}{\end{corollary}}
\newcommand{\bdfn}{\begin{definition}}   \newcommand{\edfn}{\end{definition}}
\newcommand{\benu}{\begin{enumerate}}    \newcommand{\eenu}{\end{enumerate}}
\newcommand{\bequ}{\begin{equation}}     \newcommand{\eequ}{\end{equation}}
\newcommand{\benn}{\begin{equation*}}    \newcommand{\eenn}{\end{equation*}}
\newcommand{\bexa}{\begin{example}}      \newcommand{\eexa}{\end{example}}
\newcommand{\bfig}{\begin{figure}}       \newcommand{\efig}{\end{figure}}
\newcommand{\bfra}{\begin{frame}}        \newcommand{\efra}{\end{frame}}
\newcommand{\bite}{\begin{itemize}}      \newcommand{\eite}{\end{itemize}}
\newcommand{\blem}{\begin{lemma}}        \newcommand{\elem}{\end{lemma}}
\newcommand{\bmat}{\begin{matrix}}       \newcommand{\emat}{\end{matrix}}
\newcommand{\bpma}{\begin{pmatrix}}      \newcommand{\epma}{\end{pmatrix}}
\newcommand{\bpic}{\begin{picture}}      \newcommand{\epic}{\end{picture}}
\newcommand{\bpro}{\begin{proof}}        \newcommand{\epro}{\end{proof}}
\newcommand{\bprp}{\begin{proposition}}  \newcommand{\eprp}{\end{proposition}}
\newcommand{\brem}{\begin{remark}}       \newcommand{\erem}{\end{remark}}
\newcommand{\bsub}{\begin{subequations}} \newcommand{\esub}{\end{subequations}}
\newcommand{\btab}{\begin{table}}        \newcommand{\etab}{\end{table}}
\newcommand{\btar}{\begin{tabular}}      \newcommand{\etar}{\end{tabular}}
\newcommand{\bthe}{\begin{theorem}}      \newcommand{\ethe}{\end{theorem}}
\newcommand{\bvma}{\begin{vmatrix}}      \newcommand{\evma}{\end{vmatrix}}
\newcommand{\bexe}{\begin{exercise}}      \newcommand{\eexe}{\end{exercise}}

\newcommand{\beq}{\begin{equation}}
\newcommand{\eeq}{\end{equation}}
\newcommand{\bequation}{\begin{equation}}
\newcommand{\eequation}{\end{equation}}
\newcommand{\bproof}{\begin{proof}}
\newcommand{\eproof}{\end{proof}}
\newcommand{\benumerate}{\begin{enumerate}}   \newcommand{\eenumerate}{\end{enumerate}}
\newcommand{\bitem}{\begin{itemize}}
\newcommand{\eitem}{\end{itemize}}
\newcommand{\bassumption}{\begin{assumption}}
\newcommand{\eassumption}{\end{assumption}}
\newcommand{\bprop}{\begin{proposition}}
\newcommand{\eprop}{\end{proposition}}
\newcommand{\bal}{\begin{aligned}}
\newcommand{\eal}{\end{aligned}}
\newcommand{\baligned}{\begin{aligned}}
\newcommand{\ealigned}{\end{aligned}}
\newcommand{\bseq}{\begin{subequations}}
\newcommand{\eseq}{\end{subequations}}
\newcommand{\bsubequations}{\begin{subequations}}
\newcommand{\esubequations}{\end{subequations}}
\newcommand{\bcases}{\begin{cases}}
\newcommand{\ecases}{\end{cases}}
\newcommand{\bbmatrix}     {\begin{bmatrix}}
\newcommand{\ebmatrix}     {\end{bmatrix}}
\newcommand{\btheorem}{\begin{theorem}}
\newcommand{\etheorem}{\end{theorem}}

\newcommand{\Acal}{{\cal A}}

\newcommand{\Ecal}{{\cal E}}

\newcommand{\Ical}{{\cal I}}

\newcommand{\Ncal}{{\cal N}}

\newcommand{\Zcal}{{\cal Z}}

\newcommand{\hbf}{\mathbf{h}}

\newcommand{\bcls}{\begin{columns}}
\newcommand{\bcl}[1]{\begin{column}{#1\textwidth}}
\newcommand{\ecls}{\end{columns}}
\newcommand{\ecl}{\end{column}}
